\documentclass[fleqn]{article}
\usepackage[utf8]{inputenc}
\usepackage{amsmath,amssymb,graphicx,bbm,bm,delarray,ntheorem}
\usepackage{hyperref}
\usepackage{xcolor}
\usepackage[colorinlistoftodos,prependcaption]{todonotes}

\def\Sb{\mathbf{S}}

\newcommand{\E}{\mathbb E}
\renewcommand{\S}{\mathbf S}
\renewcommand{\P}{\mathbb P}

\long\def\ignore#1{}
\def\eps{\varepsilon}

\title{Switchover phenomenon for general graphs}

\author{
Dániel Keliger\\
Department of Stochastics, Institute of Mathematics,\\
Budapest University of Technology and Economics, H-1111\\
Budapest, Hungary
\and
László Lovász
\and
Tamás Móri\\
Alfr\'ed R\'enyi Institute of Mathematics, H-1053 Budapest, Hungary
\and
Gergely \'Odor\\ Department of Network and Data Science,\\
Central European University, A-1100 Vienna, Austria}

\date{April 2023}

\begin{document}
	
\newtheorem{theorem}{Theorem}[section]
\newtheorem{prop}[theorem]{Proposition}
\newtheorem{algorithm}[theorem]{Algorithm}
\newtheorem{lemma}[theorem]{Lemma}
\newtheorem{claim}{Claim}
\newtheorem{corollary}[theorem]{Corollary}
\theorembodyfont{\rmfamily}
\newtheorem{remark}[theorem]{Remark}
\newtheorem{definition}[theorem]{Definition}
\newtheorem{example}{Example}
\newtheorem{conj}{Conjecture}
\newtheorem{prob}{Problem}
\newtheorem{step}{Step}
\newtheorem{alg}{Algorithm}

\newenvironment{proof}{\medskip\noindent{\bf Proof. }}{\hfill$\square$\medskip}
\newenvironment{proof*}[1]{\medskip\noindent{\bf Proof of #1.}}{\hfill$\square$\medskip}

	\maketitle

\tableofcontents

\begin{abstract}
We study SIR type epidemics on graphs in two scenarios: (i) when the initial
infections start from a well connected central region, (ii) when initial
infections are distributed uniformly. Previously, Ódor et al. demonstrated on a
few random graph models that the expectation of the total number of infections
undergoes a switchover phenomenon; the central region is more dangerous for
small infection rates, while for large rates, the uniform seeding is expected
to infect more nodes. We rigorously prove this claim under mild, deterministic
assumptions on the underlying graph. If we further assume that the central
region has a large enough expansion, the second moment of the degree
distribution is bounded and the number of initial infections is comparable to
the number of vertices, the difference between the two scenarios is shown to be
macroscopic.
\end{abstract}

\section{Introduction}

We study the propagation of a disease on a network, and in particular the
``switchover'' phenomenon established in \cite{OCKLK,OCKLK2}. Informally, the
phenomenon means the following. We have a network (describing the network of
interactions of people in a country), which has a denser "central region" and a
sparser "periphery". We compare the total number of nodes that get infected if
a given number of seeds (initial infections) are distributed uniformly and
randomly in the central region and in the whole graph, respectively. The
switchover phenomenon means that for a low infection rate, an epidemics
starting in the central region is worse (results in a larger epidemics), but
this switches over so that the epidemics starting uniformly over the whole
country is worse.

In \cite{OCKLK}, the authors have shown by simulation that this phenomenon occurs in many
networks (not all), and established it rigorously for some very simple
networks. In \cite{OCKLK2}, some mathematical conditions were formulated
(without proof), under which the switchover phenomenon occurs. The goal of this
paper is to generalize those results and prove them mathematically.

Our model for the spread of infection is the SIR(1) model (which is one of the
simplest). In this model, we have a finite graph $G$. A node can be in one of
three states: susceptible (S), infected (I) or resistant (R). At each step, if
a susceptible node has an infected neighbor, then it gets infected by this
neighbor with probability $\beta$. If it has several infected neighbors, then
the events that these infect the node are independent. The node becomes
infected if at least one of its infected neighbors infect it. An infected node
recovers deterministically after one step, and will be resistant from then on,
which means that it does not infect and cannot be infected. If you think of a
time scale where one step as a week, then this may be a reasonable assumption;
every event (getting infected and then passing it on) is recorded on a weekly
scale.

The main advantage of the SIR(1) model for us is that it is equivalent with a
percolation problem. A proof of this simple observation was given in
\cite{OCKLK2}. Briefly, it is not hard to see that we can decide about each
edge in advance, independently and with probability $\beta$, whether it is
going to pass on the infection, at any time when one of its endpoints is
infected and the other one is susceptible. Our model guarantees that every edge
has at most one chance to be in this situation. In other words, we keep every
edge with probability $\beta$ and delete the remaining edges; this way we get
an edge-percolated graph $G^\beta$. For a seed set $S$, we denote by
$G^\beta(S)$ the union of those components of $G^\beta$ that contain at least
one node of $S$. Then $|G^{\beta}(S)|$ nodes will be infected at one point
during the epidemic in total.

Our goal is to compare the expectations of $|G^{\beta}(\Sb_1)|$ and
$|G^{\beta}(\Sb_2)|$, where $\Sb_1$ is a random subset of the central region
and $\Sb_2$ is a random subset of the whole node set. In Section
\ref{sec:weak}, we show that (under quite general conditions) for very small
$\beta$, seeding the central region is worse, but for $\beta$ very near to $1$,
seeding the whole graph uniformly is worse (Theorem \ref{THM:EXTREME}).
However, such values of $\beta$ are unlikely to occur in real life, and also
the differences in epidemic sizes are minuscule. We call this ``weak
switchover'', and we give its formal definition in Section
\ref{sec:definitions}.

In Section \ref{sec:strong}, we formulate conditions on the graph under which we can work with values of
$\beta$ in a more reasonable range, and we can establish that the difference
between the sizes of the epidemics starting from $\Sb_1$ and $\Sb_2$ is of the
same order of magnitude as the whole graph (we call this ``strong switchover''). These
conditions on the graph (Theorem \ref{t:main}) are tighter than for weak switchover, but they
are still reasonable, and can be satisfied by real networks.

As an application of our results in Section \ref{sec:strong}, we prove that
weak switchover occurs on Chung-Lu random graphs with power-law degree
distribution \cite{chung2002connected} in Section \ref{sec:appl}. This result
was stated in \cite{OCKLK}, along with a non-rigorous proof.

\medskip

{\bf Relationship with distribution-free graph models.} Epidemics are often
studied either theoretically or by simulation on random graph models
\cite{newman2018networks}. In this paper, our goal is different: we aim to find
deterministic conditions on the graph, which give rise to the switchover
phenomenon (in expectation, where the randomness only comes from the epidemic
or percolation process). Such combinatorial results, which are studied with a
network science application in mind, are called \textit{distribution-free} in
the literature \cite{fox2020finding}. The main advantage of the
distribution-free approach is that deterministic conditions can be verified on
real networks, as opposed to the results on random graph distributions, where
we can only hope that the results also apply to real networks. Moreover, one
can go from results with deterministic conditions to results on random graphs
relatively easily (as we do in Section \ref{sec:appl}), whereas going in the
opposite direction seems much more difficult.

Proving facts that hold with high probability for random graphs for
deterministic graphs with appropriate properties goes back (at least) to the
study of quasirandom graphs \cite{CGW1989}. In the network science setting, the
study of distribution-free graph models was started by Fox et al.
\cite{fox2020finding}, and several papers followed. We refer to
\cite{roughgarden2020distribution} for a review. The deterministic constraints
studied in this topic include conditions on the triadic closure
\cite{fox2020finding}, on heterogenous degree distributions
\cite{borassi2016algorithms} and on the expansion properties
\cite{barmpalias2019idemetric} of the graphs. While one of our main conditions
is also a deterministic expansion property (a stronger one than in
\cite{barmpalias2019idemetric}), our conditions and proof techniques are
different from all previous papers that we are aware of in this topic.

\section{Results}

\subsection{Notation and setup}
\label{sec:definitions}

Let $G=(V,E)$ be a simple graph on $n$ nodes. We use the notation $|G|=|V|=n$.
For a subset $S \subseteq V$, $G(S)$ denotes the union of connected components
meeting $S$. As usual, we denote by $G[S]$ the subgraph induced by $S$.
$e(K,L)$ stands for the number of edges between $K,L \subseteq V.$ The average
degree and the second moment of the set $K \subseteq V$ is denoted by
\begin{align*}
	\overline{\deg}(K):=&\frac{1}{|K|}\sum_{v \in K}\deg(v), \\
	\overline{\deg^2}(K):=&\frac{1}{|K|}\sum_{v \in K}\deg^2(v).
\end{align*}

The path visiting vertices $v_1,v_2,\dots,v_k \in V $ is denoted by $v_1v_2
\dots v_k$. For neighboring vertices $u,v \in V$ we write $u \sim v$ and for $K
\subseteq V$, $\mathcal{N}(K)$ stands for $\{v \in V \setminus K \mid \exists u
\in K : v \sim u  \},$ i.e. the neighborhood of $K$.

We consider graphs with a specified subset $C\subseteq V$ (modeling the {\it
central region}) of size $|C|=r=cn$. Here, $0<c<1$ is considered to be
``macroscopic'', and $C$ will be denser than average in a sense to be defined
later. Throughout, we use the notation $G_1=G[C]$ and $G_2=G \setminus E(G_1)$.

For $0\le\beta\le1$,  $G^\beta$ denotes the percolation of $G$ with edge
retention probability $\beta$; in other words, the graph obtained by selecting
each edge of $G$ independently with probability $\beta$, and deleting the
unselected edges.

Usually, the set $S \subseteq V$ represents a \emph{deterministic} seed of
initial infections. We will be interested in \emph{random} seeds $\S \sim
\operatorname{Uni}(L,k)$ sampled uniformly from the $k$-subsets of a set $L
\subseteq V$ for some $k=sn$ ($0<s<c$). We think of $L$ as a macroscopic
subset; typical choices are $L=V$ and $L=C$. The corresponding random subsets
for $L=C$ and $L=V$ are $\S_C \sim \operatorname{Uni}(C,k)$ and $\S_V \sim
\operatorname{Uni}(V,k).$

In our considerations, we generate the random graph $G^\beta$ and the seed set
$\S$ independently. We let $\P_{\S}$ and $\E_{\S}$ denote the probability and
expectation if only the seed $\S$ is randomized, and define $\P_\beta$ and
$\E_\beta$ analogously when only the graph $G^\beta$ is randomized. We use no
subscript if probability and expectation are taken over both random choices.

Now we come to our two main definitions.

\begin{definition}\label{def:weak-swichover}
We say the graph $G$ exhibits a \emph{weak switchover phenomenon} with seed
size $k$ ($1\le k\le |C|$), if there are $\beta_1,\beta_2\in(0,1)$ such that
for $\S_C \sim \operatorname{Uni}(C,k)$, and $\S_V \sim
\operatorname{Uni}(V,k)$ we have
\[
\E\big(|G^{\beta_1}(\S_C)\big)> \E\big(|G^{\beta_1}(\S_V)|\big),
\]
but
\[
\E\big(|G^{\beta_2}(\S_C)\big)<\E\big(|G^{\beta_2}(\S_V)|\big).
\]
\end{definition}

Note that Definition \ref{def:weak-swichover} only requires that there is some
difference between $\E(|G^\beta(\S_C)|)$ and $\E(|G^\beta(\S_V)|)$, where this
difference could be small, even vanishing as $n \to \infty$. In a more robust
version, we require these differences to constitute a positive fraction of the
whole population. To make an exact definition, we need to consider a sequence
of graphs whose size tends to infinity:

\begin{definition}\label{def:strong_switchover}
We say the sequence of graphs $(G_n,C_n)$ exhibits a \emph{strong switchover
phenomenon} with seed sizes $k_n$, if there are real numbers $\delta>0,$
$0<\beta_1(n),\beta_2(n)<1$ such that for $\S_{n,V} \sim
\operatorname{Uni}(V(G_n),k_n)$, and $\S_{n,C} \sim
\operatorname{Uni}(C_n,k_n)$ we have
\[
\E\big(|G^{\beta_1}(\S_{n,V})\big)\geq \E\big(|G^{\beta_1}(\S_{n,C})|\big)+\delta |V(G_n)|,
\]
but
\[
\E\big(|G^{\beta_2}(\S_{n,V})\big) \leq \E\big(|G^{\beta_2}(\S_{n,C})|\big)-\delta |V(G_n)|.
\]
for large enough $n$.
\end{definition}

\subsection{Weak switchover}
\label{sec:weak}

We start by elementary remarks concerning the cases when $\beta\to 0$ and
$\beta\to~1$. It is clear that if $\beta\to 0$, then $\E(|G^\beta(S)|)\to |S|$,
while if $\beta\to 1$, then $\E(|G^\beta(S)|)\to~n$ for every nonempty set $S$
and connected $G$.

The case of small $\beta$ is straightforward, since the seeds and those nodes
reached in one step will dominate. The probability that a particular path of
length 2 is retained in $G^\beta$ is at most $\beta^2$, so with probability
$1-O(\beta^2)$, only neighbors of $S$ get infected, and each such neighbor is
infected by only one seed (here $O$ refers to $\beta\to 0$). Hence for any
subset $S\subseteq V$,
\begin{align}\label{eq:small_beta_original}
\E \left(|G^\beta(S)\right) = |S| +\beta e(S,V \setminus S)
+ O\left(\beta^2\right).
\end{align}

The asymptotics at $\beta\to 1$ is more complicated. Assume that $G$ has the
(mild) property that

\smallskip

($*$) $G$ has minimum degree $d$, it is not $d$-regular and the only edge-cuts
in $G$ with at most $d$ edges are the stars of minimum degree nodes.

\smallskip

Let $Y\subseteq V$ be the set of nodes with degree $d$. Set $\gamma=1-\beta$.
With probability at least $1-O(\gamma^{d+1})$, at most $d$ edges of $G$ are
missing in $G^\beta$. By ($*$), in this case $G^\beta$ is either a connected
spanning subgraph of $G$, or it has a single isolated node in $Y$. The
probability of the latter event is $\gamma^d$ for any given node in $Y$.

This implies that with probability at least $1-O(\gamma^{d+1})$, for every set
$S\subseteq V$, $|S|\ge 2$, the infected graph $G^\beta(S)$ will miss at most
one node in $Y\setminus S$. Hence
\begin{align}	\label{eq:small_gamma}
	\E_\beta(|G^\beta(S)|) = n - |Y\setminus S|\gamma^d + O(\gamma^{d+1}).
\end{align}

Formulas \eqref{eq:small_beta_original} and \eqref{eq:small_gamma} imply:

\begin{theorem}\label{THM:EXTREME}
Let $G$ be a connected graph, and $S_1,S_2\subseteq V$, $|S_1|=|S_2|$.

\smallskip

{\rm(a)} If $e(S_1, V \setminus S_1) >e(S_2,V \setminus S_2)$ and
$\beta$ is sufficiently close to $0$, then
$\E(|G^\beta(S_1)|)>\E(|G^\beta(S_2)|)$.

\smallskip

{\rm(b)}  If $G$ has property $(*)$, $|S_1\cap Y| > |S_2\cap Y|$, and $\beta$
is sufficiently close to $1$, then
$\E(|G^\beta(S_1)|)<\E(|G^\beta(S_2)|)$.
\end{theorem}

Coming to random seed sets, it will be easy to derive from
\eqref{eq:small_beta_original} and \eqref{eq:small_gamma} the following.

\begin{theorem}\label{THM:CENTRAL-EXTREME}
Let $G$ be a connected graph and $2\le k < r$.

\smallskip

{\rm(a)} If
\begin{align*}
\frac{r-k}{r-1}\overline{\deg}(C)> \frac{n-k}{n-1}\overline{\deg}(V),
\end{align*}
then $\E(|G^\beta(\S_C)|)>\E(|G^\beta(\S_V)|)$ if $\beta$ is
sufficiently close to $0$.

\smallskip

{\rm(b)} If $G$ has property $(*)$ and
\begin{align*}
\frac{|Y\cap C|}{r}<\frac{|Y|}{n},
\end{align*}
then $\E(|G^\beta(\S_C)|)<\E(|G^\beta(\S_V)|)$ if $\beta$ is
sufficiently close to $1$.
\end{theorem}
\begin{remark}
For fixed $c$ and small enough $s$ it is enough to assume
$\overline{\deg}(C)>\overline{\deg}(V)$ for part (a) of Theorem
\ref{THM:CENTRAL-EXTREME} as
\[
1 \leq \frac{r-1}{r-k}\frac{n-k}{n-1}=1+O(s).
\]
\end{remark}
\begin{corollary}\label{COR:CENTRAL-EXTREME}
If both conditions {\rm(a)} and {\rm(b)} above are satisfied, then $G$ exhibits
the weak switchover phenomenon for seed sets of size $k$.
\end{corollary}
Note that the both conditions say that $C$ has larger degrees than average.

In conclusion, a weak switchover phenomenon occurs for all graphs under very mild
hypotheses, but for unrealistically extreme values of $\beta$, and leading only
to minuscule differences. Our goal in the next section is to exhibit a strong
switchover with much more reasonable values of $\beta$.

\subsection{Strong switchover}
\label{sec:strong}

To establish the case of small $\beta$ for strong switchover is similar to the
analogous case for weak switchover: again seeds and their neighbors will play
the main role. We have to do more careful estimates, involving the spectrum of
$G$. Our main tool is the following refined version of \eqref{eq:small_gamma}.

\begin{lemma}\label{LEM:SMALL-BETA-STRONG}
Let $L\subseteq V$, $m=|L|$, and let $\S$ be a random
$k$-subset of $L$. Then
\begin{align*}
\E \big( |G^{\beta}(\S)| \big) = k + k \Big(\overline{\deg}(L)-\frac{k-1}{m-1} \frac{1}{m}e(L,L)\Big)\beta + R ,
\end{align*}
where
\[
|R| \le \overline{\deg^2}(V)\beta^2 n.
\]
\end{lemma}	

Applying this lemma with $L=V$ and $L=C$, we will get that for an appropriate
$\beta$, seeding the central region is substantially more dangerous than
seeding the whole node set. More exactly:

\begin{corollary}\label{COR:SMALL-BETA-MAIN}
Assume that
\begin{equation}\label{EQ:AVE-DEG}
\frac{r-k}{r-1}\overline{\deg}(C)-\frac{n-k}{n-1}\overline{\deg}(V) \geq c_1>0.
\end{equation}
Let
\begin{equation}\label{EQ:SMALL-BETA-CHOICE}
0<\beta\le \frac{1}{4} \frac{c_1}{\overline{\deg^2}(V)}s.
\end{equation}
Then
\[
\E\left (| G^\beta(\S_C)| \right)-\E\left (| G^\beta(\S_V)| \right) \ge \frac{1}{2}c_1  \beta sn.
\]
\end{corollary}

\begin{remark}
Similarly to part (a) of Theorem \ref{THM:CENTRAL-EXTREME} it is enough to
ensure that $\overline{\deg}(C)>\overline{\deg}(V)$ uniformly for
\eqref{EQ:AVE-DEG} when $s$ is small enough.
\end{remark}

Ensuring the large $\beta$ case for strong switchover is more involved and
requires further assumptions regarding the graph $G$. More precisely, we assume
edge expansion of the central region instead of large average degree.

\begin{definition}
We say that a graph $G=(V,E)$ has {\it edge-expansion} $(a,q)$ with some $a>0$
and $0<q<\frac{1}{2}$, if for every set $X\subset V$, $qn <|X|\le n/2$, the
number of edges between $X$ and $V\setminus X$ is at least $a|X|$.
\end{definition}

\begin{remark}
Note that the parameters $a,q$ might not be optimal. If $a_1 \leq a_2, q_1 \leq q_2$ and $G$ has edge-expansion $(a_2,q_1)$, than it is also true that $G$ has edge-expansion $(a_1,q_2)$.
\end{remark}

The following lemma shows that if the central region has a large enough
expansion, the epidemic will produce more infections from a uniform seeding
when $\beta$ is close to $1$.

\begin{lemma}	\label{LEM:LARGE-BETA-STRONG}
Let $b$ be the average degree of nodes of $V\setminus C$ in $G$, and assume
$G_1$ has edge expansion $(a,q)$ with $ q<1/3.$ Then
\begin{align}
\label{eq:lemma_main}
\begin{split}
&\E(|G^\beta(\S_V)|)-\E(|G^\beta(\S_C)|) \geq \\
& s(1-c)(1-\beta)^b n - \frac{c}{c-s} qn - \left( 1+\frac{c}{c-s}\right)n\rho^r-ne^{-2c k/3},
\end{split}
\end{align}
where $\rho:=\left(\frac{e(1-\beta)^a}{q} \right)^q.$

\end{lemma}

\begin{remark}
\label{r:large_beta}

When $G_1$ has edge expansion $(a,q)$ with $a>b$ and $q=(1+\varepsilon)e(1-\beta)^a$ for some $\varepsilon>0$ we end up with $0 \leq \rho <1$ resulting in
\begin{align*}
\left( 1+\frac{c}{c-s}\right)n\rho^r+ne^{-2c k/3}=o(n)
\end{align*}
for all fixed $0<\beta<1$. Furthermore, as $0 \leq b<a$ it is possible to set $0<\beta<1$ to a value for which $q<\frac{1}{3}$ and
\begin{align*}
s(1-c)(1-\beta)^b  >  \frac{c}{c-s}(1+\varepsilon) e(1-\beta)^a =\frac{c}{c-s} q,
\end{align*}
thus, there is a $\delta>0$ such that $\E(|G^\beta(\S_V)|)>\E(|G^\beta(\S_C)|)+\delta n$ for large enough $n$.
\end{remark}

Our main result concerning strong switchover will easily follow by a
combination of Lemmas \ref{LEM:SMALL-BETA-STRONG} and
\ref{LEM:LARGE-BETA-STRONG}.

\begin{theorem}\label{t:main}
Let $(G_m:~m=1,2,\dots)$ be a sequence of such
that $n_m=|V(G_m)|\to\infty$. Let $C_m\subseteq V(G_m)$ so that $|C_m|=c_mn_m$ and let $b_m$ denote the
average degree in $G_m$ of nodes in $V(G_m)\setminus C_m$ with some uniform bound $b_{m} \leq b_{max}$.

Assume there is a $\varepsilon>0$ such that :
\begin{itemize}
	\item $ c_m \leq1-\varepsilon$,
	\item $\eps \leq s_m \leq (1-c_m)c_m/2$,
	\item For any $0<q<\frac{1}{3}$ $G_m[C_m]$ has edge-expansion $(b_{max}+\varepsilon,q)$ when $m$ is large enough.
\end{itemize}
Also, assume  the second moments of the degrees are uniformly bounded.

Then the graph sequence $((G_m,C_m):~m=1,2,\dots)$ exhibits the strong
switchover phenomenon for seed sizes $s_mn_m$.
\end{theorem}

\section{Proofs}

We start with a simple identity, which will imply that when choosing a "small"
random seed $\S$ compared to $V$, it becomes unlikely that two vertices in $\S$
are adjacent to each other, and hence
\[\E \left(e(\S,V  \setminus \S)\right)\approx \E \left(e(\S,V) \right)=k \overline{\deg}(L).
\]
\begin{lemma}\label{l:V_minus_S}
	Let $\S$ be a random $k$-element subset of $L$. Then
	\[
	\E_\S \left( e(\S,V \setminus \S)\right)
	= k \Big(\overline{\deg}(L)-\frac{k-1}{m-1} \frac{1}{m}e(L,L)\Big).
	\]
\end{lemma}

The last (error) term can be estimated as
\[
\frac{k-1}{m-1} \frac{1}{m}e(L,L) < \frac{k}{m} \overline{\deg}(L).
\]

\begin{proof}
	\begin{align*}
		&\E \left(e(\S, V \setminus \S) \right)=\E \left(e(\S, V ) \right)-\E \left(e(\S,  \S) \right)=\\
		&\frac{k}{m}e(L,V)-\frac{k(k-1)}{m(m-1)}e(L,L)=k \Big(\overline{\deg}(L)-\frac{k-1}{m-1} \frac{1}{m}e(L,L)\Big)
	\end{align*}
	
	This proves the lemma.
\end{proof}

\subsection{Weak switchover}

\begin{proof*}{Theorem \ref{COR:CENTRAL-EXTREME}}
	We start with the small $\beta$ case. 	
	\begin{align*}
		&\E\left( \left |G^{\beta}(\S_C) \right|\right)-\E\left( \left
		|G^{\beta}(\S_V) \right|\right)=\\
		& \left(\E\left(e\left(\S_C,V \setminus \S_C \right) \right)
        -\E\left(e\left(S_V, V\setminus \S_v \right)\right) \right)\beta	+O\left(\beta^2\right)
	\end{align*} 	
	Due to Lemma \ref{l:V_minus_S} the leading term can be bounded as
	\begin{align*}
		&\E\left(e\left(\S_C,V \setminus \S_C \right) \right)-\E\left(e\left(\S_V,V \setminus \S_V \right) \right)=\\
		& k \left(\overline{\deg}(C)-\frac{k-1}{r-1}
		\frac{1}{r}e(C,C)-\overline{\deg}(V)+\frac{k-1}{n-1}\frac{1}{n}e(V,V) \right)  \geq \\
		& k \left(\frac{r-k}{r-1} \overline{\deg}(C)-\frac{n-k}{n-1}\overline{\deg}(V) \right)>0,
	\end{align*}
	making 	$\E\left( \left |G^{\beta}(\S_V) \right|\right)>\E\left( \left
	|G^{\beta}(\S_C) \right|\right)$ for sufficiently small $\beta.$ 	
	
	As for small $\gamma=1-\beta$
	\begin{align*}
		&\E \left(\left|G^{\beta}(\S_V) \right| \right)-\E \left(\left|G^{\beta}(\S_C) \right| \right)
		=\\
		&\left(\E \left(|Y \setminus \S_C| \right)-\E \left(|Y \setminus \S_V |\right) \right)\gamma^{d}
        +O \left(\gamma^{d+1} \right)=\\
		&k\left[ \left(1-\frac{|Y \cup C|}{|C|}\right)-\left(1-\frac{|Y|}{|V|} \right) \right] \gamma^{d}
        +O \left(\gamma^{d+1} \right) =\\
		&k \left(\underbrace{\frac{|Y|}{|V|}-\frac{|Y \cup C|}{|C|}}_{>0} \right)\gamma^{d}+O \left(\gamma^{d+1} \right),
	\end{align*}
	implying $\E \left(\left|G^{\beta}(\S_V) \right| \right)>\E
	\left(\left|G^{\beta}(\S_C) \right| \right) $ when $\beta$ is close to $1$.
\end{proof*}

\subsection{Strong switchover}

\subsubsection{Lemmas for small $\beta$}
We are going to prove Lemma \ref{LEM:SMALL-BETA-STRONG} in the following
slightly stronger form:

\begin{lemma}\label{LEM:SMALL-BETA-STRONG-X}
Let
$L\subseteq V$, $m=|L|$, and let $\S$ be a random $k$-subset of $L$. Then
\begin{align*}
\E \big( |G^{\beta}(\S)| \big) = k + k \Big(\overline{\deg}(L)-\frac{k-1}{m-1} \frac{1}{m}e(L,L)\Big)\beta + R ,
\end{align*}
where
\[
 - \frac{1}{2}\left(\overline{\deg^2}(V)-\overline{\deg}(V) \right)\beta^2 n
 \le R \le \left(\overline{\deg^2}(V)-\overline{\deg}(V) \right)\beta^2 n.
\]
\end{lemma}

\begin{proof}
For ease of notation introduce $\deg_{S}(v):=e \left(\{v\},\S) \right)$
representing the number of neighbors of vertex $v \in V$ from $\S \sim
\operatorname{Uni}(L,k).$		For the lower bound on $R$, it suffices to
count nodes in $\S$ and their neighbors:
\begin{align}\label{EQ:SMALL-BETA-LOWER}
\E_\S(|G^{\beta}(\S)|) &\ge k + \sum_{v\in \mathcal{N}(\S)}\big(1-(1-\beta)^{\deg_\S(v)}\big)\nonumber\\
&\ge k + \sum_{v\in \mathcal{N}(\S)}\left(\beta\deg_\S(v)- \beta^2\binom{\deg_\S(v)}{2}\right)\nonumber\\
& =  k + \beta e(\S,V \setminus \S)  -\beta^2\sum_{v\in V}\binom{\deg_{\S}(v)}{2}.
\end{align}
Note that, by definition $\binom{\deg_{\S}(v)}{2}=0$ when $\deg_{\S}(v) \leq 1.$

The probability that the random set $\S$ contains two given nodes in L is $\frac{k(k-1)}{m(m-1)}$, hence
\begin{align*}
 &\beta^2\sum_{v\in V}\E \left[ \binom{\deg_{\S}(v)}{2} \right]
 = \frac{k(k-1)}{m(m-1)}\beta^2\sum_{v\in V}\binom{\deg_{L}(v)}{2} \leq \\
 &\frac{1}{2}\beta^2 \sum_{v \in V} \deg(v) \left(\deg(v)-1 \right)=\frac{1}{2}
 \left(\overline{\deg^2}(V) -\overline{\deg}(V)\right)\beta^2 n.
\end{align*}

For the upper bound notice

\begin{align*}
\E_{\S} \left(\left|G^{\beta}(\S) \right| \right)
= k+\sum_{v \in \mathcal{N}(\S)}\P_{\S}(v \in G^{\beta}(\S)) +\sum_{v \in V \setminus (\S \cup \mathcal{N}(\S))}
\P_{\S}(v \in G^{\beta}(\S)).
\end{align*}

Let $\deg_{\S}^{\beta}(v)$ denote number of neighbors of $v  \in V$ from $\S$
in the percolated graph $G^\beta$. Clearly, for $v \in \mathcal{N}(\S)$
\begin{align*}
  &\P_{\S}\left( v \in G^\beta(\S) \right)
  \leq \P_{\S} \left(\deg_{\S}^{\beta}(v)>0 \right)
  +\P_{\S} \left( \left. v \in G^\beta(\S) \right| \deg_{\S}^{\beta}(v)=0 \right)= \\
  &\underbrace{1-\left(1-\beta \right)^{\deg_{\S}(v)}}_{ \leq \beta \deg_{\S}(v) }+\P_{\S}
  \left( \left. v \in G^\beta(\S) \right| \deg_{\S}^{\beta}(v)=0 \right)
\end{align*}

Let $H$ denote the graph where the edges between $v$ and $\S$ are deleted.
Since edge retention happens independently
$$\P_{\S} \left( \left. v \in G^\beta(\S) \right| \deg_{\S}^{\beta}(v)=0 \right)=\P_{\S}(v \in H^{\beta}(\S)). $$

We will call a length $2$ path $vuw$ \emph{proper} if $v$ and $w$ has distance
two, or in other words, $v,u,w$ does not form a triangle. $\mathcal{A}_{v}(G)$
denotes the event that there is no proper path starting from $v$ in the
percolated graph $G^\beta.$ Since $H$ is a subgraph of $G$ $\mathcal{A}_{v}(G)
\subseteq \mathcal{A}_{v}(H).$

Observe that vertices $v \in V \setminus (\S \cup \mathcal{N}(\S))$ in graph
$G$ and $v \in \mathcal{N}(\S)$ in $H$ are at least $2$ steps away from the set
$\S$. This implies
\begin{align*}
v \in V \setminus (\S \cup \mathcal{N}(\S)) \ \ \P_{\S} \left(v \not \in G^\beta(\S) \right)  \geq & \P \left(\mathcal{A}_{v}(G) \right), \\
v \in \mathcal{N}(\S) \ \ \P_{\S}(v \not \in H^{\beta}(\S)) \geq & \P_{\S}(\mathcal{A}_{v}(H)) \geq \P \left(\mathcal{A}_{v}(G) \right).
\end{align*}
Together, they make bound
\begin{align*}
\E_{\S} \left( \left|G^{\beta}(\S) \right| \right)
\leq k+\beta e \left( \S, V \setminus \S \right)+\sum_{v \in V \setminus \S}
\left(1-\P \left(\mathcal{A}_{v}(G) \right) \right).
\end{align*}

Let $\delta(u)$ denote the number of $vuw$ proper paths for some $w$. Clearly,
$\delta(u) \leq \deg(u)-1$.

Note that two proper paths $vuw, \ vu'w'$ can only share an edge in their $vu,
\,vu'$ segment when $u=u'$, the second segment is always disjoint. ($w'=u, \,
u'=w$ would make $u,v,w$ a triangle.) This means any two proper paths are
independent when $u \neq u'$.

\begin{align*}
&\P\left( \bigcup_{w \sim u} \left \{vuw \ \textit{is a proper path in} \ G^{\beta}  \right \}\right)=\beta \left(1-(1-\beta)^{\delta(u)} \right)
\end{align*}
\begin{align*}
\P \left(\mathcal{A}_{v}(G)\right)=& \prod_{u \sim v} \left[1-\beta \left(1-(1-\beta)^{\delta(u)} \right) \right]  \overset{*}{\geq } 1 -\sum_{u \sim v} \beta \left(1-(1-\beta)^{\delta(u)} \right) \\
\geq &1-\beta^2 \sum_{u \sim v} \delta(u) \geq 1-\beta^2 \sum_{u \sim v} \left(\deg(u)-1 \right)
\end{align*}
\begin{align*}
\sum_{v \in V \setminus \S} \left(1-\P \left(\mathcal{A}_{v}(G) \right) \right) \leq& \beta^2 \sum_{v \in V} \sum_{u \sim v}(\deg(u)-1)=\beta^2 \sum_{u} \deg(u) \left( \deg(u)-1\right)\\
=&\left(\overline{\deg^2}(V)-\overline{\deg}(V)\right)\beta^2 n
\end{align*}

Note that at step $*$ we used the union bound for independent events.
\end{proof}

\begin{proof}(Corollary \ref{COR:SMALL-BETA-MAIN})

Let $R_C,R_V$ be the remainder terms in Lemma \ref{LEM:SMALL-BETA-STRONG} when $L=C,V$. Since $\frac{n-k}{n-1} \geq \frac{r-k}{r-1}, $\eqref{EQ:AVE-DEG} implies $\overline{\deg}(C) \geq \overline{\deg}(V).$ Thus,
\begin{align*}
|R_C|, |R_V| \leq \overline{\deg^2}(V)\beta^2 n
\end{align*}

This results in the bound
\begin{align*}
\E&\left (| G^\beta(\S_C)| \right)-\E\left (| G^\beta(\S_V)| \right)\\
&= \Big(\overline{\deg}(C)-\frac{k-1}{r-1}
\frac{1}{r}e(C,C)-\overline{\deg}(V)+\frac{k-1}{n-1}\frac{1}{n}e(V,V)\Big)\beta k +  R_C-R_V \\
&\ge \Big(\frac{r-k}{r-1} \overline{\deg}(C)-\frac{n-k}{n-1}\overline{\deg}(V)\Big)\beta k- 2 \overline{\deg^2}(V)\beta^2 n\\
&\geq  \frac{1}{2}c_1\beta k=\frac{1}{2}c_1\beta sn.
\end{align*}
\end{proof}

\subsubsection{Lemmas for large $\beta$}

\begin{lemma}\label{LEM:EXPAND}
Let $G$ be a graph with $n$ nodes and edge-expansion $(a,q)$, where $a>1$ and
$q<1/3$. Let $0<\beta<1$, and let $H$ be a largest connected component of
$G^\beta$. Then
\begin{equation*}
\P\big(|H|\leq (1-q)n\big)\leq  \rho^n,
\end{equation*}
where
\begin{equation}\label{EQ:TAU-DEF}
\rho=\Big(\frac{e(1-\beta)^a}{q}\Big)^q.
\end{equation}
\end{lemma}

\noindent For the bound to be nontrivial, we need that $q>e(1-\beta)^a$.

We start with an elementary observation. 	

\begin{claim}\label{CLAIM:LARGE-COMP}
If the largest connected component of $G^\beta$ has at most $n-t$ nodes, where
$t\le n/3$, then there is a set $X\subseteq V$ such that $t \le |X|\le n/2$ and
no edge of $G^\beta$ connects $X$ and $V\setminus X$.
\end{claim} 	

\begin{proof}(Claim \ref{CLAIM:LARGE-COMP})
Indeed, let $H_1$ be the nodeset of the largest connected component of $G^\beta$.
Then $|H|\le n-t$ by hypothesis. If $|H|\ge n/2$, then $X=V\setminus H$
satisfies the conditions in the claim. So suppose that $|H|<n/2$. If $t\le
|H|$, then $H$ satisfies the conditions in the claim. So suppose that $|H|<t$.
Let us add further connected components to $H$ as long as it remains at most
$n/2$ in cardinality, to get a set $X$. If $|X|\ge t$ then we are done, so
suppose that $|X|< t$. Adding any other connected component, we get a set $X'$
with $|X'|>n/2$ and $|X'|<|X|+t$. If $|X'|\le n-t$, then $V\setminus X'$
satisfies the conditions in the claim. So suppose that $|X'|>n-t$. But then
$n-t<|X'|\le |X|+t \le 2t$, and so $t>n/3$, contrary to the hypothesis.
\end{proof}

\begin{proof}(Lemma \ref{LEM:EXPAND})
Let $z=(1-\beta)^a$. For a fixed $k$-subset $X$ ($qn\le k\le n/2$), the graph
$G$ has at least $ak$ edges between $X$ and $V\setminus X$, and the probability
that none of them is selected is at most $(1-\beta)^{ak}=z^k$. So the
probability that there is a set $X\subseteq V$ with $q n\le |X|\le n/2$ and
having no edges between $X$ and $V\setminus X$ is at most
\begin{align*}
\sum\limits^{\lfloor n/2 \rfloor}_{k=\lceil qn \rceil} \binom{n}{k} z^k.
\end{align*}
Let $p=z/(1+z)$ and let $\xi$ be a $\mathrm{Binom}(n,p)$ distributed random
variable. Then, by the well-known Chernoff--Hoeffding bound,
\begin{align}\label{eq:Chernoff-Hoeffding}
 & \sum\limits^{\lfloor n/2 \rfloor}_{k=\lceil qn \rceil} \binom{n}{k}z^k = (1+z)^n
\sum\limits^{\lfloor n/2 \rfloor}_{k=\lceil qn \rceil} \binom{n}{k} p^k (1-p)^{n-k}\nonumber\\
&\leq  (1+z)^{\!n}\,\P(\xi\ge qn) \le(1+z)^n\left[
\left(\frac{p}{q}\right)^{\!q}\left(\frac{1-p}{1-q}\right)^{\!1-q}
\right]^{\!n}\nonumber\\
&=\left[\left(\frac{z}{q}\right)^q\frac1{(1-q)^{1-q}}\right]^n.
\end{align}
Here
\[
\frac{1}{(1-q)^{1-q}}=\left(1+\frac{q}{1-q}\right)^{\!1-q} <e^{q},
\]
hence, by \eqref{eq:Chernoff-Hoeffding} and Claim \ref{CLAIM:LARGE-COMP}
\[
\P\big(|H|\leq(1-q)n\big)< \left(\frac{e z}{q}\right)^{qn},
\]
proving the lemma.
\end{proof}

\begin{proof}(Lemma \ref{LEM:LARGE-BETA-STRONG})

Let $H_1$ denote the component of $G^\beta$ with largest number of nodes in
$C$, and let $H_2,\dots,H_{m}$ be the other components. Note that $H^C
\subseteq H_1$ where $H^C$ is the largest component in $G^{\beta}(C)$.

Let $h_j=|H_j|$ and
$c_j = |C \cap H_j|$. Let $p_j$ and $q_j$ denote the probability that $H_j$ is
not infected by $\S_V$ and $\S_C$, respectively. Then
\[
p_j = \prod_{i=0}^{k-1} \left(1-\frac{h_j}{n-i}\right),
 \]
where the last inequality holds whenever $h_j\le n-k+1$; else, $p_j=0$.
Similarly,
\[
q_j= \prod_{i=0}^{k-1} \left(1-\frac{c_j}{r-i}\right),
\]
where again the last inequality holds if $c_j\le r-k+1$ and $0$ otherwise. Then
\begin{align}\label{EQ:GAIN}
\E_{\beta}(|G^\beta(\S_V)|)&-\E_{\beta}(|G^\beta(\S_C)|) = \sum_{j=1}^{m} h_j(1-p_j) - \sum_{j=1}^{m} h_j(1-q_j)\nonumber\\
&= \sum_{j=1}^{m} h_j(q_j-p_j).
\end{align}

The main idea of the proof is that we partition the index set $K=\{1, \dots,
m\}$ into four sets:
\begin{align*}
    K_1 &:= \{1\},\\
    K_2 &:= \{j\in K:~h_j=1,c_j=0\},\\
    K_3 &:= \{j\in K \setminus K_1:~h_j\le c_j/(c-s)\},\\
    K_4 &:= K\setminus K_1\setminus K_2\setminus K_3.
\end{align*}
Let $V_i=\cup_{j\in K_i}H_j$.

We lower bound the sum in equation \eqref{EQ:GAIN} using a different estimate
over each set $K_j$. There are two sets where uniform seeding is more dangerous
($K_2$ and $K_4$), one set where the two seedings are essentially equally
dangerous ($K_1$, the giant component) and one set where the central seeding is
more dangerous ($K_3$), but this set $K_3$ only contains components which have
a relatively large part in $C$ compared to $V\setminus C$, and since the giant
component is quite large in $G_1$, the components in $K_3$ cannot have too much
weight
We make this intuition precise in the computation below. 	

First, we fix the percolation $G^\beta$, and estimate the expectations over the
choice of seed sets. We start with $K_1$, which only contains the index of the
component with the largest number of nodes in $C$. This component will have a
non-empty intersection with both $\S_V$ and $\S_C$ with high probability. More
exactly,
\begin{align}		\label{EQ:K1S}
\E_\beta(|V_1\cap G^\beta(\S_V)|&-\E_\beta(|V_1\cap G^\beta(\S_C)|= h_1(q_1-p_1) \ge -h_1 p_1\nonumber\\
&\ge - h_1 \left(1-\frac{h_1}{n}\right)^k \ge -n\left(1-\frac{h_1}{n}\right)^k\ge -n e^{-h_1 k/n}.
\end{align}

Next we consider $K_2$, the index set of those components of $G^\beta$ that are
isolated nodes of $V\setminus C$. Clearly $q_j=1$ and
$p_j=\prod_{i=0}^{k-1}\left(1-\frac{1}{n-j}\right) =\frac{n-k}{n}$ for $j\in K_2$.
So
\begin{align}
\label{EQ:K2S}
&\E_\beta(|V_2\cap G^\beta(\S_V)|)-\E_\beta(|V_2\cap G^\beta(\S_C)|)
=\sum_{j \in K_2} h_j(q_j-p_j) =\frac{k}{n}\left|V_2 \right|.
\end{align}
For $K_3$ we use the lower bound
\begin{align}\label{EQ:K3S}
\E_\beta(|V_3\cap G^\beta(\S_V)|)&-\E_\beta(|V_3\cap G^\beta(\S_C)|)=
\sum_{j \in K_3} h_j (q_j-p_j) > - \sum_{j \in K_3} h_j \nonumber\\
&\ge - \sum_{j \in K_3} \frac{c_j}{c-s} \geq - \frac1{c-s} |C\setminus V_1|.
\end{align}
Finally, if $j\in K_4$, then it must satisfy
\begin{equation*}		
1-\frac{h_j}{n} \le 1-\frac{c_j}{r-k},
\end{equation*}
which implies that for these components $q_i\ge p_i$ and so
\begin{equation}\label{EQ:K4S}
\E_\beta(|V_3\cap G^\beta(\S_V)|)-\E_\beta(|V_3\cap G^\beta(\S_C)|)\ge 0.
\end{equation}

Summing \eqref{EQ:K1S}--\eqref{EQ:K4S}, we get
\begin{align}\label{ES-MAIN}
\E_\beta(|G^\beta(\S_V)|)&-\E_\beta(|G^\beta(\S_C)|)\nonumber\\
&\ge -ne^{-|V_1| k/n} + \frac{k}{n}|V_2| - \frac1{c-s} |C\setminus V_1|.
\end{align}

To compute the expectation of this over the percolation, let us denote the
degree of the $i^{th}$ node in $V\setminus C$ by $b_i$. Then by Jensen's
inequality (since $(1-\beta)^x$ is convex),
\begin{equation}\label{eq:jensen}
\E_\beta(|V_2|)= \sum_{i=1}^{n-r} (1-\beta)^{b_i} \ge  (n-r)(1-\beta)^b.
\end{equation}

Recall $H^C$. By applying Lemma~\ref{LEM:EXPAND} to $G_1$, we have
$|V_1|=|H_1|\geq |H_1^C| > (1-q)r$ with probability at least $1-\rho^r$, where
$\rho$ is defined by \eqref{EQ:TAU-DEF}. Hence,
\[
\E\big(e^{-|V_1| k/n}\big)\le   e^{-(1-q)c k} + \rho^r\le e^{-2c k/3} + \rho^r,
\]
and
\begin{align*}
\E \left( \left|H^C \right| \right)    \geq (1-q)r \P \left(\left|H^C \right|>(1-q)r \right) \geq (1-q)(1-\rho^r)r
\end{align*}
resulting in
\begin{align}
\nonumber
&\E\left(|C\setminus V_1|\right)\le \E \left( \left|C \setminus H^C \right| \right)=r-\E \left( \left|H^C  \right | \right) \leq \\
\label{EQ:C_minus_V_1}
& r \left[1-(1-q)(1-\rho^r)\right] \leq qr+r \rho^r.
\end{align}

To sum up,
\begin{align*}
\E&(|G^\beta(\S_V)|)-\E(|G^\beta(\S_C)|)\nonumber\\
&\ge -ne^{-2c k/3} + \frac{k}{n} \left(1-\frac{k}{2n} \right)(n-r)(1-\beta)^b - \frac1{c-s} qr - \Big(n+\frac{r}{c-s}\Big)\rho^r\nonumber\\
&\ge  s(1-c)(1-\beta)^b n - \frac{c}{c-s} qn - \left( 1+\frac{c}{c-s}\right)n\rho^r-ne^{-2c k/3}.
\end{align*}
This proves the lemma.
\end{proof}

\begin{proof*}{Theorem \ref{t:main}}
We start with the small $\beta$ case. We need the following fact:

\begin{claim}
\label{c:degrees}
Let $H$ be a graph with $N$ nodes and edge-expansion $(a,q)$ ($a\ge 1,
0<q<1/2$). Then the average degree in $H$ is at least $2a\frac{N-1}{N+1}$.
\end{claim}

\begin{proof}(Claim \ref{c:degrees})
	
We check this for the case when $|V(H)|=2m+1$ is odd (the even case is
similar). For every $m$-subset $S\subseteq V$, there are at least $am$ edges
between $S$ and $V\setminus S$. This gives $am\binom{2m+1}{m}$ edges. Each edge
is counted $2\binom{2m-1}{m-1}$ times, hence
\[
|E(H)|\ge am\frac{\binom{2m+1}{m}}{2\binom{2m-1}{m-1}} = a\frac{(2m+1)m}{m+1}.
\]
Thus the average degree is
\[
\overline{\deg}(H) = \frac{2|E(H)|}{2m+1} = 2a\frac{m}{m+1}.
\]
This proves the Claim.
\end{proof}

Consider a graph $G_m$ from the given sequence. In the rest of this proof, we
omit the indices $m$, to make the arguments more readable. The Claim above
implies that

\begin{align*}
\overline{deg}(C) \geq 2a \frac{r-1}{r+1} \sim 2 a,
\end{align*}
therefore
\begin{align*}
&\frac{r-k}{r-1}\overline{\deg}(C)-\frac{n-k}{n-1}\overline{\deg}(V)
\sim \left(1-\frac{s}{c}\right) \overline{\deg}(C)-(1-s)\overline{\deg}(V) \geq \\
&\left(1-\frac{s}{c} \right)\overline{\deg}(C)-\left(c \overline{\deg}(C)+(1-c)b \right)= \\
&\left(1-c-\frac{s}{c} \right) \overline{\deg}(C)-(1-c)b \overset{s \leq \frac{1}{2}c(1-c)}{\geq} \\
&(1-c) \left(\frac{1}{2}\overline{\deg}(C)-b \right) \gtrsim (1-c)(a-b) \geq \varepsilon^2 \Rightarrow \\
&\frac{r-k}{r-1}\overline{\deg}(C)-\frac{n-k}{n-1}\overline{\deg}(V)
\geq   \varepsilon^2-o(1) \geq \frac{1}{2}\varepsilon^2=:c_1>0
\end{align*}
when $n$ is large enough. This means the conditions of Corollary
\ref{COR:SMALL-BETA-MAIN} are satisfied when $\beta$ is small enough.

The large beta case is an easy consequence of Remark \ref{r:large_beta}.
\end{proof*}

\section{Application: Chung-Lu model with power law degree distribution}
\label{sec:appl}

In this section, we apply Lemma \ref{LEM:LARGE-BETA-STRONG} to rigorously prove
the previous claim of~\cite{OCKLK}, that the uniform seeding can be more
dangerous in random graphs with power-law degree distribution with exponent
$\tau \in (2,3)$ if
\begin{equation}
\label{eq:OCKLK}
\frac{1}{n}\beta^{-\frac{1}{|\tau-3|}}  \ll s \ll \beta^{\frac{\tau-1}{3-\tau}}.
\end{equation}
In $\cite{OCKLK}$, this claim  appears as an if and only if statement, however,
in this section we only address the ``if'' part. Our proof strategy has already
been outlined in a previous paper \cite{OCKLK2}, but this is the first time
when we give a fully rigorous proof.

We start by defining the random graph distribution in the focus of this
section, which is one of the most standard models of networks with a power-law
degree distribution.

\begin{definition}
\label{def:CL} Let us denote by $\mathcal{CL}(\tau)$ the \emph{distribution of
Chung-Lu random graphs with exponent $\tau \in (2,3)$}, where the nodes $v_i
\in V$ are indexed from 1 to $n$, and $v_i$ and $v_j$ are connected by an edge
independently with probability $p_{ij}=\min \left \{\frac{d_id_j}{D}, 1 \right
\}$, where $d_i=\left( \frac{n}{i} \right)^{\frac{1}{\tau-1}}$ and
$D=\sum_{k=1}^n d_k$.
\end{definition}

Notice that the expected degree of a node with index $i$ in a graph sampled
from $\mathcal{CL}(\tau)$ is $d_i + o(1)$. Moreover, for  $d_i$ to be less than
some integer degree $d$, we need to have $ \frac{i}{n} \le  d^{1-\tau},$ which
hints that the exponent of the cumulative degree distribution is expected to be
around $1-\tau$, and therefore the degree distribution is expected to follow a
power-law with exponent $\tau \in (2,3)$. The average degree of the
distribution is expected to be a constant, because
\begin{align}
\label{eq:constant_deg}
\sum_{k=1}^n d_k \sim  \int_{1}^n \left( \frac{n}{x} \right)^{\frac{1}{\tau-1}} dx
=  \frac{\tau-1}{\tau-2} \left(n -n^{\frac{1}{\tau-1}} \right) =  \Theta(n).
\end{align}
Chung-Lu random graphs were introduced in \cite{chung2002connected}, we refer
to this paper and follow-up works for more precise statements on interpreting
Definition \ref{def:CL}. Here, we continue by stating an elementary result
about the edge expansion of Chung-Lu random graphs, which may have already
appeared in the literature in a similar form, but we are not aware of it.
\begin{lemma}
\label{lem:CL} If $G$ is sampled from $\mathcal{CL}(\tau)$ with $\tau \in
(2,3)$, and if $C$ is the set of vertices of $G$ with index $i \le cn$, with
$\frac{1}{\sqrt{n}}\ll c \ll \left( \log (n)\right)^{-\frac{\tau-1}{3-\tau}} $,
then $G[C]$ has edge expansion $\left(\frac{n}{4D}
c^{-\frac{3-\tau}{\tau-1}},0\right)$ asymptotically almost surely.
\end{lemma}

Before proving Lemma \ref{lem:CL}, let us show how we can apply it to
rigorously prove the claim in equation \eqref{eq:OCKLK}, at least partially.

\begin{corollary}
\label{corollary:CL}
Let us consider a sequence of graphs $G_n$ sampled from $\mathcal{CL}(\tau)$
with $\tau \in (2,3)$, and let us define the central region $C$ as the $\lfloor cn \rfloor$
nodes with the largest expected degree (i.e., index) in $G_n$. Let us assume that the size of the seed set satisfies $s=\Theta(c)$ and $c-s=\Theta(c)$. Under the mild assumption $1-\beta = \Theta(1)$, if
\begin{align}
\label{eq:range}
 \sqrt{\frac{\log(n)}{n}} \ll s  &\ll \left(\frac{\beta}{ \log(n)} \right)^{{\frac{\tau-1}{3-\tau}}}
\end{align}
hold then the uniform seeding creates a larger epidemic than the central
seeding ($G_n$ has the weak switchover property) with probability tending to 1
as $n \rightarrow \infty$.
\end{corollary}

Notice that with this choice of parameters, the number of seeds is linear only
the size of the central region, but sub-linear in the size of the graph.

As shown in Figure \ref{fig_appl}, the parameter ranges set by equation
\eqref{eq:range} form a subset of the parameters in equation \eqref{eq:OCKLK},
therefore, Corollary \ref{corollary:CL} is weaker than the claim in
\cite{OCKLK}. To generalize Corollary \ref{corollary:CL} for the remaining
parameter ranges, different proof methods are necessary.

\begin{figure*}[htb]
    \centering
    \includegraphics[width = .8\textwidth]{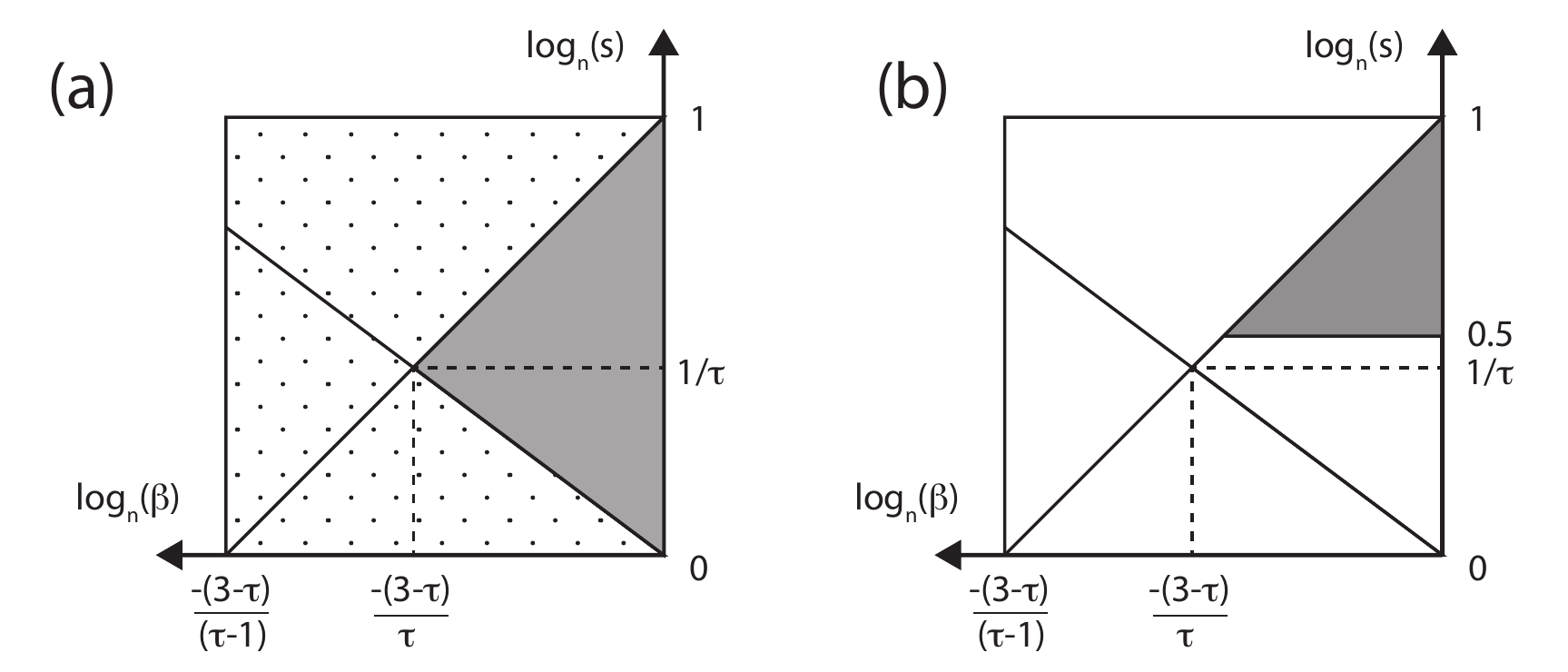}
    \caption{Phase diagrams for the switchover phenomenon on Chung-Lu random
    graphs with power-law degree distribution ($\tau \in (2,3)$).
    (a) With grey we show the region where the central area is more dangerous,
    as claimed in \cite{OCKLK}, and with the dotted pattern we show the region
    where the central area is more dangerous by equation~\eqref{eq:OCKLK},
    as claimed in \cite{OCKLK}.
    (b) With grey we show the region where the central area is more dangerous
    by Corollary \ref{corollary:CL}.}
\label{fig_appl}
\end{figure*}

\begin{proof*}{Corollary \ref{corollary:CL}}

We will use Lemma \ref{LEM:LARGE-BETA-STRONG} for Chung-Lu random graphs with
\begin{equation}
\label{eq:newq}
q=\frac{s(c-s)(1-c)(1-\beta)^b}{2c}.
\end{equation}
For Lemma \ref{LEM:LARGE-BETA-STRONG} to be applicable, we need to make sure that for our choice
\begin{equation}
\label{eq:q_check}
q=\frac{s(c-s)(1-c)(1-\beta)^b}{2c} >(1+\varepsilon)e(1-\beta)^a,
\end{equation}
and to satisfy equation \eqref{eq:lemma_main}, we need that
\begin{equation}
\label{eq:condition_check}
\frac{1}{2} s(1-c)(1-\beta)^b >\left( 1+\frac{c}{c-s}\right)\rho^r+ e^{-2c k/3},
\end{equation}
keeping in mind that $c$ and $s$ are not constants anymore. Condition
$q<\frac{1}{3}$ is trivially satisfied for large enough $n$ as $q=\Theta(s).$

Recall, that we chose $s=\Theta(c)$ and $c-s=\Theta(c)$. Notice that we can
apply Lemma \ref{lem:CL}, because the condition $\frac{1}{\sqrt{n}}\ll c \ll
(\log(n))^{-\frac{\tau-1}{3-\tau}} $ holds by \eqref{eq:range}. Then, since we
also know $b=\Theta(1)$ by equation \eqref{eq:constant_deg}, we can show that
equation \eqref{eq:q_check}  holds if
\begin{align*}
\log(c) \gg \log(1-\beta) {}c^{-\frac{3-\tau}{\tau-1}},
\end{align*}
which is implied by equation \eqref{eq:range} as
\begin{align*}
\log(1-\beta) c^{-\frac{3-\tau}{\tau-1}} \leq&-\beta  c^{-\frac{3-\tau}{\tau-1}}
=-\left(\frac{c}{\beta^{\frac{\tau-1}{3-\tau}}} \right)^{-\frac{3-\tau}{\tau-1}} \stackrel{\eqref{eq:range}}{\ll} - \log(n) \\
=& 2 \log \left(\frac{1}{\sqrt{n}} \right) \ll \log (c).
\end{align*}
Similarly, equation \eqref{eq:condition_check} holds if
\begin{align*}
\rho^r+e^{-2c k/3}  \ll s.
\end{align*}
By equation \eqref{eq:constant_deg} and substititing the definition of $\rho$
from equation \eqref{EQ:TAU-DEF}, we get
\begin{align*}
(1+\varepsilon)^{-qr} +e^{-2c k/3} \ll s,
\end{align*}
which must hold because equation \eqref{eq:range} implies
\begin{align*}
\Theta(qr) =  \Theta(ck) = \Theta\left(s^2 n \right) \stackrel{\eqref{eq:range}}{\gg} \log(n).
\end{align*}

Therefore, Lemma \ref{lem:CL} implies that for these parameter ranges, the
uniform seeding can be more dangerous, and weak switchover occurs.
\end{proof*}

We conclude the section by providing the proof of Lemma \ref{lem:CL}.

\begin{proof*}{Lemma \ref{lem:CL}}
For $S \subset V$ with $|S| \le |C| /2$, let $X_S$ be the number of edges
between $S$ and $C\setminus S$. In the first part of the proof, we show that
$\E(X_S)  \geq \frac{n}{2D} c^{-\frac{3-\tau}{\tau-1}} |S|$ for every $S$, and
in the second part we prove that the variables $X_S$ are all well-concentrated
around their expectation.

Note that since we assumed $c \gg \frac{1}{\sqrt{n}}$, and since we know
$D=\Theta(n)$, we have that $\min \left \{d_{\lfloor cn \rfloor}^2,D \right \}
=d_{\lfloor cn \rfloor}^2 \geq c^{-\frac{2}{\tau-1}}$. Then, we compute the
expectation of $X_S$ as

\begin{align*}
\E (X_S)=& \sum\limits_{i \in S} \sum\limits_{i\in C\setminus S} \min \left\{\frac{d_id_j}{D}, 1 \right \}= \sum\limits_{i \in S} \sum\limits_{i\in C\setminus S} \frac{d_id_j}{D} \ge \sum\limits_{i \in S} \sum\limits_{i\in C\setminus S} \frac{d_{\lfloor cn \rfloor}^2}{D}  \\
 \ge& \frac{|S| (|C|-|S|)}{D}  c^{-\frac{2}{\tau-1}} \overset{|S| \leq \frac{|C|}{2}}{=} \frac{n}{2D} c^{-\frac{3-\tau}{\tau-1}}|S|.
\end{align*}

Next, we use the union bound, and well-known multiplicative Chernoff bounds on binomial random variables, to prove that the random variables $X_S$ are concentrated around their expectation. We bound

\begin{align*}
 \P\left(\exists S \subset C, |S| \le \frac{|C|}{2} \text{ with } X_S \le \frac12\E(X_S) \right)
 &\le \sum\limits_{\substack{S \subset C \\ |S| \le \frac{|C|}{2}}} \P \left(X_S \le \frac12\E(X_S) \right) \\
& \le \sum\limits_{\substack{S \subset C \\ |S| \le \frac{|C|}{2}}} e^{-\frac18 \E(X_S) } \\
&= \sum\limits_{\substack{S \subset C \\ |S| \le \frac{|C|}{2}}} e^{- \frac{n}{16D}  c^{-\frac{3-\tau}{\tau-1}}|S|} .
\end{align*}

Set $\eta = \frac{n}{16D}  c^{-\frac{3-\tau}{\tau-1}}$. Let us change the
indexing of the sum to the size of the set $S$, and apply a standard upper
bound on binomial coefficients to obtain
\begin{equation*}
\sum\limits_{\substack{S \subset C \\ |S| \le \frac{|C|}{2}}} e^{- \eta |S|}
= \sum\limits_{k=1}^{\frac{\lfloor cn \rfloor}{2}} \binom{\lfloor cn \rfloor}{k} e^{- \eta k}
\le \sum\limits_{k=1}^{\frac{\lfloor cn \rfloor }{2}}  \left( \frac{enc}{k} \right)^k e^{- \eta k}
\le \sum\limits_{k=1}^{\frac{\lfloor cn \rfloor }{2}}  \left( cn e^{1- \eta} \right)^k.
\end{equation*}
Notice that we have arrived at a geometric series with common ratio $cn e^{1-
\eta}$, which tends to zero as long as $\eta = \frac{n}{16D}
c^{-\frac{3-\tau}{\tau-1}}  \gg \log(n)$; an asymptotic inequality that holds
by the assumption $c \ll \left( \log (n) \right)^{-\frac{\tau-1}{3-\tau}}$.
Therefore, we arrived to the equation
$$ \P\left(\exists S \subset C, |S| \le \frac{C}{2} \text{ with } X_S \le
\frac{n}{4D} c^{-\frac{3-\tau}{\tau-1}}|S|  \right) \rightarrow 0,$$
which completes the proof of the lemma.
\end{proof*}

\section{Concluding remarks}

In this paper, we gave the first fully rigorous proofs of the switchover
phenomenon, introduced in \cite{OCKLK}, for general classes of graphs. We
showed that weak switchover exists under mild conditions on the graph, and we
also showed sufficient conditions for strong switchover.

One limitation of the current paper is that, in the case of the strong
switchover, the size of the seed set was assumed to be fairly large, of size
$\Omega(n)$. Although for the Chung-Lu model in Section \ref{sec:appl} we did
study smaller seed sets, and we did use the machinery of the strong switchover
proofs, we were only able to show the existence of the weak switchover
phenomenon. This agrees with the simulations and heuristic derivations of the
previous work \cite{OCKLK}, which also claimed that Chung-Lu models exhibit
weak switchover, but not strong switchover. However,\cite{OCKLK} also showed
that the strong switchover phenomenon occurs with much smaller seed sets on
geometric graphs, notably on the commuting network of Hungary constructed from
real data, and also random graph models with an underlying geometry.
Unfortunately, our current results do not say much about such geometric graphs.

As in this paper, proving the existence of strong switchover with small seed
sets would boil down to the distribution of component sizes in the percolated
graph $G^\beta$. However, contrary to this paper, we need the existence of at
least medium size components also in the periphery, because if most of the
peripheral nodes in $G^\beta$ are contained in bounded-size components, then we
need $\Omega(n)$ seeds to have strong switchover (even to have an epidemic of
size $\Omega(n)$). Finding appropriate conditions for such medium size
components in the periphery which could lead to the existence of strong
switchover with small seed sets (say, of size $\sqrt{n}$ or even $\log n$) is
an interesting future direction.

\medskip

{\bf Acknowledgment.} The authors are thankful to Marianna Bolla for her
insightful remarks. This work has been supported by the Dynasnet ERC Synergy
project (ERC-2018-SYG 810115). Gergely Ódor was supported by the Swiss National
Science Foundation, under grant number P500PT-211129.


\begin{thebibliography}{99}


\bibitem{barmpalias2019idemetric}
G.~Barmpalias, N.~Huang, A.~Lewis-Pye, A.~Li, X.~Li, Y.~Pan and T.~Roughgarden:
The idemetric property: when most distances are (almost) the same. {\em
Proceedings Of The Royal Society A}. \textbf{475}, 20180283 (2019)

\bibitem{borassi2016algorithms}
M.~Borassi: Algorithms for metric properties of large real-world networks from
theory to practice and back. (IMT School for Advanced Studies Lucca,2016)

\bibitem{chung2002connected}
F.R.K.~Chung and L.~Lu: Connected components in random graphs with given
expected degree sequences. {\em Annals Of Combinatorics}. \textbf{6} (2002).
125-145

\bibitem[1989]{CGW1989}
F.R.K.~Chung, R.L.~Graham and R.M.~Wilson: Quasi-random graphs, {\it
Combinatorica} {\bf 9} (1989), 345--362.

\bibitem{fox2020finding}
J.~Fox, T.~Roughgarden, C.~Seshadhri, F.~Wei and N.~Wein: Finding cliques in
social networks: A new distribution-free model. {\em SIAM Journal On
Computing}. \textbf{49}, 448-464 (2020)

\bibitem{newman2018networks}
M.~Newman: Networks. (Oxford university press, 2018)

\bibitem{OCKLK}
G.~\'Odor, D.~Czifra, J.~Komj\'athy, L.~Lov\'asz and M.~Karsai: Switchover
phenomenon induced by epidemic seeding on geometric networks, {\it Proc.\ Nat.\
Acad.\ Sci.} (2021) {\bf 118} (41) e2112607118.

\bibitem{OCKLK2}
G.~\'Odor, D.~Czifra, J.~Komj\'athy, L.~Lov\'asz and M.~Karsai: Longer-term
seeding effects on epidemic processes: a network approach, {\it Scientia\ et\
Securitas.} (2022) {\bf 2} (4) 409--417.

\bibitem{roughgarden2020distribution}
T.~Roughgarden and C.~Seshadhri: Distribution-Free Models of Social Networks.
{\em ArXiv Preprint ArXiv:2007.15743}. (2020)

\end{thebibliography}
\end{document}